\documentclass[reqno]{amsart}
\usepackage[margin=1.2in]{geometry}
\usepackage[backref=page]{hyperref}
\usepackage{cite}
\usepackage{enumitem}
\usepackage{xcolor}

\newcommand{\rev}[1]{#1}
\newcommand{\add}[1]{#1}

\theoremstyle{plain}
\newtheorem{theorem}{Theorem}
\newtheorem{proposition}[theorem]{Proposition}
\newtheorem{corollary}[theorem]{Corollary}
\newtheorem{lemma}[theorem]{Lemma}

\theoremstyle{definition}

\theoremstyle{remark}
\newtheorem{remark}{Remark}

\allowdisplaybreaks

\begin{document}

	\author{Phuong Le}
	\address{Phuong Le$^{1,2}$ (ORCID: 0000-0003-4724-7118)\newline
		$^1$Faculty of Economic Mathematics, University of Economics and Law, Ho Chi Minh City, Vietnam; \newline
		$^2$Vietnam National University, Ho Chi Minh City, Vietnam}
	\email{phuongl@uel.edu.vn}

	\subjclass[2020]{35J61, 35J75, 35B06, 35B09}
	\keywords{semilinear elliptic equation, singular nonlinearity, half-space, monotonicity, rigidity}

	\title[Singular semilinear elliptic equations]{Singular semilinear elliptic equations in half-spaces}
	\begin{abstract}
		We prove the monotonicity of positive solutions to the problem $-\Delta u = f(u)$ in $\mathbb{R}^N_+ := \{(x',x_N)\in\mathbb{R}^N \mid x_N>0 \}$ under zero Dirichlet boundary condition with a possible singular nonlinearity $f$. In some situations, we can derive a precise estimate on the blow-up rate of $\frac{\partial u}{\partial\eta}$ as $x_N \to 0^+$, where $(\eta,e_N)>0$, and obtain a classification result. The main tools we use are the method of moving planes and the sliding method.
	\end{abstract}

	\maketitle

	\section{Introduction}

	The monotonicity and symmetry of solutions to the semilinear elliptic problem
	\begin{equation}\label{main}
		\begin{cases}
			-\Delta u = f(u) &\text{ in } \mathbb{R}^N_+,\\
			u > 0 &\text{ in } \mathbb{R}^N_+,\\
			u = 0 &\text{ on } \partial\mathbb{R}^N_+,
		\end{cases}
	\end{equation}
	where
	\[
	\mathbb{R}^N_+ := \{x:=(x',x_N)\in\mathbb{R}^N \mid x_N>0 \},
	\]
	are well studied in the literature. \add{Throughout the paper, we assume that $N\ge2$.} Berestycki, Caffarelli, and Nirenberg \cite{MR1395408,MR1655510} demonstrated that if \( f:[0,+\infty) \to \mathbb{R} \) is a Lipschitz function with \( f(0) \ge 0 \), then any classical solution of \eqref{main} is monotone in the \( x_N \)-direction. When \( f \) is only locally Lipschitz continuous on $[0,+\infty)$, a similar monotonicity result can be obtained for solutions that are bounded on all strips $\Sigma_\lambda := \{(x',x_N)\in\mathbb{R}^N \mid 0<x_N<\lambda \}$ ($\lambda>0$), as shown in \cite{MR4142367,MR2254600}. The case where \( f(0) < 0 \) is more complex, and a complete proof of monotonicity for solutions in this scenario is currently only available for dimension \( N = 2 \), as detailed in \cite{MR3593525,MR3641643}. For results on symmetry of solutions, which is usually called rigidity in the literature, we refer to \cite{MR1470317,MR1655510,MR794096,MR937538,MR1260436} and the references therein.

	In this paper, we are mainly interested in problem \eqref{main} with singular nonlinearity at zero in the sense that $f:(0,+\infty)\rev{\to\mathbb{R}}$ is locally Lipschitz continuous and $\lim_{t\to0^+} f(t) = +\infty$. A model problem is given by
	\begin{equation}\label{singular}
		\begin{cases}
			-\Delta u = \dfrac{1}{u^\gamma} + g(u) &\text{ in } \mathbb{R}^N_+,\\
			u > 0 &\text{ in } \mathbb{R}^N_+,\\
			u = 0 &\text{ on } \partial\mathbb{R}^N_+,
		\end{cases}
	\end{equation}
	where $\gamma>0$ and $g:[0,+\infty)\rev{\to\mathbb{R}}$ is a locally Lipschitz continuous function.
	\add{It is} well established that solutions to problem \eqref{singular} are generally not smooth up to the boundary. In fact, \add{when $\gamma>1$,} the gradient of solutions becomes unbounded \add{near} the boundary\add{; see \cite{MR1037213} for bounded domains, and Theorem \ref{th:monotonicity2} below for the precise blow-up rate in the half-space}. Given the natural regularity behavior of these solutions (as discussed in \cite{MR427826}), we focus on solutions \( u \in C^2(\mathbb{R}^N_+) \cap C(\overline{\mathbb{R}^N_+}) \) to \eqref{main}. Consequently, the equation is well-defined in the classical sense within the domain's interior.

	Since the seminal paper \cite{MR427826}, singular semilinear elliptic problems
	\begin{equation}\label{bounded}
		\begin{cases}
			-\Delta u = \dfrac{1}{u^\gamma} + g(u) &\text{ in } \Omega,\\
			u > 0 &\text{ in } \Omega,\\
			u = 0 &\text{ on } \add{\partial}\Omega,
		\end{cases}
	\end{equation}
	where $\Omega\subset\mathbb{R}^N$ is a bounded domain,	have been extensively studied from various perspectives. We specifically reference the works \cite{MR2718666,MR2927112,MR3105927,MR2592976,MR2099611,MR1194224,MR1037213,MR3800107}, which are closely related to our research. A key focus in the study of these equations is understanding the behavior of solutions near the boundary, where they often lose regularity. A generalized version of the \add{Hopf} boundary lemma was obtained in \cite{MR3912757}. The symmetry of solutions was studied in \cite{MR3800107} (see also \cite{MR4333974,MR4658655} and the references therein).

	As demonstrated in \cite{MR3912757}, to obtain the \add{Hopf} boundary lemma for \eqref{bounded}, one may exploit a scaling argument near the boundary which leads to the study of a limiting problem in the half-space
	\begin{equation}\label{pure}
		\begin{cases}
			-\Delta u = \dfrac{1}{u^\gamma} &\text{ in } \mathbb{R}^N_+,\\
			u > 0 &\text{ in } \mathbb{R}^N_+,\\
			u = 0 &\text{ on } \partial\mathbb{R}^N_+,
		\end{cases}
	\end{equation}
	which is exactly problem \add{\eqref{singular}} with $g\equiv0$. Solutions to problem \eqref{pure} have been classified recently in elegant papers \cite{MR4753083,2024arXiv240403343M}. These results reveal that all weak solutions to \eqref{pure} with $\gamma>1$ must be either of the form
	\begin{equation}\label{pure_sol1}
	u(x) \equiv \frac{(\gamma+1)^\frac{2}{\gamma+1}}{(2\gamma-2)^\frac{1}{\gamma+1}}	x_N^\frac{2}{\gamma+1}
	\end{equation}
	or of the form
	\begin{equation}\label{pure_sol2}
	u(x) \equiv \lambda^{-\frac{2}{\gamma+1}} v(\lambda x_N)
	\end{equation}
	where $\lambda>0$ and $v \in C^2(\mathbb{R}_+) \cap C(\overline{\mathbb{R}_+})$ is the unique solution to
	\[
	\begin{cases}
		-v'' = \dfrac{1}{v^\gamma}, & t>0,\\
		v(t)>0, & t>0,\\
		v(0)=0, ~ \lim_{t\to+\infty} v'(t) = 1.
	\end{cases}
	\]
	\add{Moreover, \eqref{pure} has no solution if $0<\gamma\le1$; see \cite{2024arXiv240403343M}.}
	\add{For positive classical solutions continuous up to the boundary, one-dimensional symmetry for every continuous, nonnegative, nonincreasing $f$ in any dimension $N\ge2$ was established earlier by Porretta and V\'eron \cite[Theorem~III]{PV04}. The more recent works \cite{MR4753083,2024arXiv240403343M} also address additional estimates and generalized weak solutions.}

	One notable feature of problem \eqref{pure} is that its nonlinearity is decreasing on $(0,+\infty)$. Hence the weak comparison principle holds in large subdomains of $\mathbb{R}^N_+$ and the monotonicity of solutions follows directly from this principle. In this paper, we consider a more general situation by studying the monotonicity and rigidity results for solutions \( u \in C^2(\mathbb{R}^N_+) \cap C(\overline{\mathbb{R}^N_+}) \) to \eqref{main} with a possible singular nonlinearity $f$. In particular, we address the issue that $f$ is not decreasing in the whole $(0,+\infty)$. The main assumption on $f$ we  require is the following:
	\begin{enumerate}[label=\textnormal{($F$)},ref=$F$]
		\item\label{F} for any $M>0$, there exists $C(M)>0$ such that
		\[
		f(s) - f(t)\le C(M)(s-t) \quad\text{ for all } 0<t\le s\le M.
		\]
	\end{enumerate}
	Our first result is the following monotonicity result, which holds in a very general setting. Indeed, we require only the behavior of $f$ near its possible singular point.
	\begin{theorem}\label{th:monotonicity}
		Assume that $f:(0,+\infty)\rev{\to\mathbb{R}}$ is a locally Lipschitz continuous function satisfying \eqref{F} and there exist $c_0,t_0>0$ such that
		\[
		f(t) > c_0t \quad\text{ for all } 0<t<t_0.
		\]
		Let \( u \in C^2(\mathbb{R}^N_+) \cap C(\overline{\mathbb{R}^N_+}) \) be a solution to \eqref{main} with $u\in L^\infty(\Sigma_\lambda)$ for all $\lambda>0$.
		Then
		\[
		\frac{\partial u}{\partial x_N} > 0 \quad\text{ in } \mathbb{R}^N_+.
		\]
	\end{theorem}
	Theorem \ref{th:monotonicity} applies not only to singular nonlinearities but also \add{to nonsingular nonlinearities satisfying the same lower bound near zero, such as $f(t)=t+t^p$ with $p>1$}.
	Since $f$ is not decreasing, we need to derive a weak comparison principle for the problem in narrow strips and exploit the moving plane method to prove Theorem \ref{th:monotonicity}. Similar results for singular problems in bounded domains were obtained in \cite{MR3800107,MR3912757}. Theorem \ref{th:monotonicity} can be applied to problem \eqref{singular}, provided that the solution is bounded in every strip. \add{When $\gamma>1$, even under the weaker assumption that the solution is bounded in only one strip, we can obtain precise estimates for the blow-up rate of its derivatives near the boundary, as shown in the next result.}
	{\color{black}
	\begin{theorem}\label{th:monotonicity2} Assume that $\gamma>1$ and $g:[0,+\infty)\to\mathbb{R}$ is a locally Lipschitz continuous function.
		Let \( u \in C^2(\mathbb{R}^N_+) \cap C(\overline{\mathbb{R}^N_+}) \) be a solution to \eqref{singular} with $u\in L^\infty(\Sigma_{\overline\lambda})$ for some $\overline\lambda>0$.
		Then there exist $c_2,\lambda_1>0$ such that
		\begin{equation}\label{gradient_upper}
			|\nabla u(x)| < c_2 x_N^\frac{1-\gamma}{\gamma+1} \quad\text{ in } \Sigma_{\lambda_1}.
		\end{equation}
		Moreover, for each $\beta\in(0,1]$, there exist $c_1,\lambda_0>0$ such that
		\begin{equation}\label{gradient_blowup}
			\frac{\partial u(x)}{\partial \eta} > c_1 x_N^\frac{1-\gamma}{\gamma+1} \quad\text{ in } \Sigma_{\lambda_0}
		\end{equation}
		for all $\eta\in \mathbb{S}^{N-1}_+$ with $(\eta,e_N) \ge \beta$, where $\mathbb{S}^{N-1}_+:=\mathbb{R}^N_+ \cap \partial B_1(0)$ and $e_N:=(0,\dots,0,1)$. \add{In particular, $\frac{\partial u}{\partial x_N}>0$ in a strip $\Sigma_{\lambda_0}$. If, in addition, $u\in L^\infty(\Sigma_\lambda)$ for all $\lambda>0$, then}
		\[
		\frac{\partial u}{\partial x_N} > 0 \quad\text{ in } \mathbb{R}^N_+.
		\]
	\end{theorem}
	}
	\rev{Estimates \eqref{gradient_upper} and \eqref{gradient_blowup} show that the gradient of $u$ blows up exactly at the rate $x_N^{\frac{1-\gamma}{\gamma+1}}$ as $x_N\to0^+$, and the upper bound \eqref{gradient_upper} holds for the full gradient, with constants independent of the direction.} \add{In particular, choosing $\eta=e_N$ in \eqref{gradient_blowup}, we obtain the two-sided estimate
	\[
	c_1 x_N^\frac{1-\gamma}{\gamma+1} < |\nabla u(x)| < c_2 x_N^\frac{1-\gamma}{\gamma+1} \quad\text{ in } \Sigma_{\min\{\lambda_0,\lambda_1\}}.
	\]}

	In this paper, we also exploit the techniques and ideas from \cite{MR4753083} to establish one-dimensional symmetry of solutions to singular problems \add{of which problem \eqref{singular} is a model}, where $\gamma>1$ and $g:(0,+\infty)\to\add{[0,+\infty)}$ is a locally Lipschitz continuous function \add{satisfying \eqref{F}, nonincreasing in a neighborhood of $+\infty$ and} such that $\limsup_{t\to+\infty} t^\gamma g(t) < +\infty$. Notice that $g$ may have a singularity at zero such as $g(t)=\frac{1}{t^\beta}$ with $\beta\ge\gamma$.
	\begin{theorem}\label{th:rigidity}
		Assume that $\gamma>1$ and $f:(0,+\infty)\rev{\to\mathbb{R}}$ is a positive locally Lipschitz continuous function satisfying \eqref{F} and
		\begin{enumerate}
			\item[(i)] there exists $c_0>0$ such that
			\[
			f(t) > \frac{c_0}{t^\gamma} \quad\text{ for all } t>0,
			\]
			\item[(ii)] there exist $c_1,t_1>0$ such that $f$ is nonincreasing on $(t_1,+\infty)$ and
			\[
			f(t) < \frac{c_1}{t^\gamma} \quad\text{ for all } t>t_1.
			\]
		\end{enumerate}
		Let \( u \in C^2(\mathbb{R}^N_+) \cap C(\overline{\mathbb{R}^N_+}) \) be a solution to \eqref{main} with $u\in L^\infty(\Sigma_{\overline\lambda})$ for some $\overline\lambda>0$.
		Then $u(x)\equiv v(x_N)$, where $v$ is given by the formula
		\[
		\int_{0}^{v(t)} \frac{ds}{\sqrt{M+F(s)}} = \sqrt{2} t \quad\text{ for all } t\ge0
		\]
		for some $M\ge0$, where $F(s) = \int_{s}^{+\infty} f(t) dt$.
	\end{theorem}
	We will employ the sliding method, which was introduced by Berestycki and Nirenberg \cite{MR1159383}, to prove Theorem \ref{th:rigidity}.
	We stress that in Theorem \ref{th:rigidity} we do not assume that $f$ is nonincreasing in the whole domain $(0,+\infty)$. If $f$ is nonincreasing throughout $(0,+\infty)$, one-dimensional symmetry without strip-boundedness is already guaranteed by \cite[Theorem~III]{PV04}. For completeness, we record the resulting classification by an elementary ODE argument and include a direct Kelvin-transform proof. \add{In particular, the absence of a strip-boundedness assumption for $-\Delta u=e^{-2u}$ follows from the earlier theorem of Porretta and V\'eron; compare \cite{DLS2026}.}
	\begin{theorem}\label{th:rigidity2}
		Assume that $f:(0,+\infty)\rev{\to\mathbb{R}}$ is a nonincreasing positive locally Lipschitz continuous function.
		Let \( u \in C^2(\mathbb{R}^N_+) \cap C(\overline{\mathbb{R}^N_+}) \) be a solution to \eqref{main}.
		Then $u$ depends only on $x_N$. Consequently, such a solution exists if and only if $\int_1^{+\infty} f(t) dt<+\infty$. Moreover, when such a solution exists, it is given by $u(x)\equiv v(x_N)$, where $v$ is determined by the formula
		\begin{equation}\label{formula}
		\int_{0}^{v(t)} \frac{ds}{\sqrt{M+F(s)}} = \sqrt{2} t \quad\text{ for all } t\ge0
		\end{equation}
		for some $M\ge0$, where $F(s) = \int_{s}^{+\infty} f(t) dt$.
	\end{theorem}

	{\color{black}
	\begin{remark}\label{rem:unify}
		Formula \eqref{formula} unifies the two types of one-dimensional solutions to \eqref{pure} described above. Let $f(t)=t^{-\gamma}$. Since $\int_1^{+\infty}t^{-\gamma}dt<+\infty$ if and only if $\gamma>1$, Theorem \ref{th:rigidity2} recovers the nonexistence of solutions to \eqref{pure} for $0<\gamma\le1$. If $\gamma>1$, then $F(s)=\frac{s^{1-\gamma}}{\gamma-1}$. For $M=0$, formula \eqref{formula} reads
		\[
		\frac{2\sqrt{\gamma-1}}{\gamma+1}\, v(t)^\frac{\gamma+1}{2} = \sqrt{2}\, t,
		\quad\text{that is,}\quad
		v(t) = \frac{(\gamma+1)^\frac{2}{\gamma+1}}{(2\gamma-2)^\frac{1}{\gamma+1}}\, t^\frac{2}{\gamma+1},
		\]
		which is exactly the solution \eqref{pure_sol1}. For $M>0$, differentiating \eqref{formula} gives $(v')^2 = 2(M+F(v))$, so that $\lim_{t\to+\infty}v'(t)=\sqrt{2M}$. On the other hand, if $v$ denotes the solution with $\lim_{t\to+\infty}v'(t)=1$ appearing in \eqref{pure_sol2}, then $v_\lambda(t):=\lambda^{-\frac{2}{\gamma+1}}v(\lambda t)$ solves the same equation, vanishes at $t=0$, and satisfies $\lim_{t\to+\infty}v_\lambda'(t)=\lambda^\frac{\gamma-1}{\gamma+1}$. Since the quantity $\frac12(v_\lambda')^2-F(v_\lambda)$ is constant and $F(v_\lambda(t))\to0$ as $t\to+\infty$, $v_\lambda$ satisfies \eqref{formula} with $M=\frac12\lambda^\frac{2(\gamma-1)}{\gamma+1}$. Hence the solutions with $M>0$ are precisely the solutions \eqref{pure_sol2}, with
		\[
		\lambda = (2M)^\frac{\gamma+1}{2(\gamma-1)}.
		\]
		In other words, $M=\frac12\lim_{t\to+\infty}(v'(t))^2$ is the parameter of the family, and the solution \eqref{pure_sol1} is the limiting member of the family \eqref{pure_sol2} as $\lambda\to0^+$.
	\end{remark}
	}

	{
	\color{black}
		\begin{corollary}\label{cor:exp}
			Let $N\ge2$. Every solution $u\in C^2(\mathbb{R}^N_+)\cap C(\overline{\mathbb{R}^N_+})$ of
			\[
			\begin{cases}
			-\Delta u=e^{-2u} &\text{in }\mathbb{R}^N_+,\\
			u>0 &\text{in }\mathbb{R}^N_+,\\
			u=0 &\text{on }\partial\mathbb{R}^N_+
			\end{cases}
			\]
			is one-dimensional and is exactly one of the following functions:
			\[
			u(x)=\ln(1+x_N),
			\]
			or, for some $L>0$,
			\[
			u(x)=\ln\sinh\big(L(x_N+t_L)\big)-\ln L,
			\qquad
			t_L:=\frac1L\ln\big(L+\sqrt{L^2+1}\big).
			\]
			No boundedness assumption on $u$ in any strip is required. Moreover, in view of the regularity result in \cite[Theorem~1.2]{DLS2026}, the same classification holds for weak solutions in the sense of \cite{DLS2026}, again without the strip-boundedness assumption imposed there when $N\ge3$.
		\end{corollary}

		\begin{proof}
			Apply Theorem \ref{th:rigidity2} with $f(t)=e^{-2t}$. Then
			\[
			F(s)=\int_s^{+\infty}e^{-2t}\,dt=\frac12e^{-2s}.
			\]
			For $M=0$, formula \eqref{formula} gives $v(t)=\ln(1+t)$. If $M=L^2/2>0$, the first integral becomes
			\[
			(v')^2=L^2+e^{-2v}.
			\]
			Setting $y=e^v$ yields $y'=\sqrt{1+L^2y^2}$ and $y(0)=1$, whence
			\[
			y(t)=\frac1L\sinh\big(L(t+t_L)\big),
			\qquad
			t_L=\frac1L\operatorname{arsinh}L
			=\frac1L\ln\big(L+\sqrt{L^2+1}\big).
			\]
			This gives exactly the two families above. The final assertion follows because \cite[Theorem~1.2]{DLS2026} proves, without assumption (A), that every weak solution considered there belongs to $C^2(\mathbb{R}^N_+)\cap C(\overline{\mathbb{R}^N_+})$.
		\end{proof}

		\begin{remark}\label{rem:exp}
			As already follows from \cite[Theorem~III]{PV04}, the classification in Corollary \ref{cor:exp} does not require the strip-boundedness assumption appearing in the higher-dimensional symmetry part of \cite{DLS2026}. More precisely, for $N\ge3$, Dai, Liu, and Sciunzi assume that $u$ is bounded above in some strip in order to derive upper bounds for $u$ and a global bound for $|\nabla u|$. These estimates are then used in their translation argument and in the maximum principle on an unbounded domain; see \cite[Sections~2.2 and 3.1]{DLS2026}. Their separate Kelvin-transform argument without this assumption is carried out only in dimension $N=2$; see \cite[Section~3.2]{DLS2026}.

			The direct proof included here likewise avoids these estimates. Since $u$ is continuous at the origin, the Kelvin transform $\hat u$ tends to zero at infinity, so the comparison functions used there have compact support in $\Sigma_\lambda$, and the argument works in every dimension $N\ge2$.
		\end{remark}
	}

	\add{For quasilinear counterparts of problem \eqref{main} with singular nonlinearities, we refer to \cite{Le2024quasi}.}
	The rest of this paper is devoted to the proofs of our main results. In Section \ref{sect2} we prove a comparison principle for narrow strips and derive some a priori estimates for solutions. In Section \ref{sect3}, we prove the monotonicity and rigidity of solutions by means of the moving plane and sliding methods.

	\section{Preliminaries}\label{sect2}

	\subsection{Weak comparison principle for narrow strips}
	Motivated by \cite{MR4142367}, we prove a weak comparison principle which can be applied to problems with singular nonlinearities.
	\begin{proposition}\label{prop:wcp}
		Assume that $f:(0,+\infty)\rev{\to\mathbb{R}}$ is a locally Lipschitz continuous function such that \eqref{F} holds and \( u \in C^2(\mathbb{R}^N_+) \cap C(\overline{\mathbb{R}^N_+}) \) satisfies
		\[
		\begin{cases}
			-\Delta u \le f(u), ~ u > 0 &\text{ in } \mathbb{R}^N_+,\\
			u = 0 &\text{ on } \partial\mathbb{R}^N_+,\\
			u \in L^\infty(\Sigma_{\overline\lambda}) &\text{ for some } \overline\lambda>0.
		\end{cases}
		\]
		Then there exists a small positive number $\lambda^*=\lambda^*(f, \|u\|_{L^\infty(\Sigma_{\overline\lambda})})<\overline{\lambda}$ such that: if $0<\lambda\le\lambda^*$ and \( v \in C^2(\mathbb{R}^N_+) \cap C(\overline{\mathbb{R}^N_+}) \) satisfies
		\[
		\begin{cases}
			-\Delta v \ge f(v), ~ v > 0 & \text{ in } \Sigma_\lambda,\\
			v > 0 & \text{ on } \partial\mathbb{R}^N_+,\\
			u \le v & \text{ on } \{x_N=\lambda\},\\
		\end{cases}
		\]
		then $u\le v$ in $\Sigma_\lambda$.
	\end{proposition}

	\begin{proof}
		In what follows, we consider $\lambda<\overline\lambda$.
		For $R>0$, let $\varphi_R \in C^\infty(\mathbb{R}^{N-1})$ be such that
		\begin{equation}\label{varphi}
			\begin{cases}
				0 \le \varphi_R \le 1 &\text{ in } \mathbb{R}^{N-1},\\
				\varphi_R = 1 &\text{ in } B'_R,\\
				\varphi_R = 0 &\text{ in } \mathbb{R}^{N-1} \setminus B'_{2R},\\
				|\nabla\varphi_R| \le \frac{2}{R} &\text{ in } B'_{2R} \setminus B'_R,
			\end{cases}
		\end{equation}
		where $B_r'$ denotes the ball in $\mathbb{R}^{N-1}$ of center $0'\in\mathbb{R}^{N-1}$ with radius $r$. We set
		\[
		\varphi(x) = w^+(x) \varphi_R^2(x') \chi_{\Sigma_\lambda}(x),
		\]
		where $w^+ := \max\{u - v, 0\}$.
		Since the support of $\varphi$ is compactly contained in $\Sigma_\lambda\cup\{x_N=\lambda\}$, we can use it as a test function in $-\Delta u \le f(u)$ and $-\Delta v \ge f(v)$. Then subtracting, we obtain
		\begin{align*}
			\int_{\Sigma_\lambda} |\nabla w^+|^2 \varphi_R^2 &\le -2\int_{\Sigma_\lambda} (\nabla w^+, \nabla \varphi_R) w^+ \varphi_R + \int_{\Sigma_\lambda} (f(u)-f(v)) w^+ \varphi_R^2\\
			&\le 2\int_{\Sigma_\lambda} |\nabla w^+| |\nabla \varphi_R| w^+ \varphi_R + \int_{\Sigma_\lambda} (f(u)-f(v)) w^+ \varphi_R^2.
		\end{align*}
		In the set $\Sigma_\lambda\cap\{w^+>0\}$ we have
		\[
		0 < v < u \le \|u\|_{L^\infty(\Sigma_{\overline\lambda})}.
		\]
		Hence by exploiting Young's inequality and using \eqref{F}, we have
		\begin{align*}
			\int_{\Sigma_\lambda} |\nabla w^+|^2 \varphi_R^2 &\le \frac{1}{2} \int_{\Sigma_\lambda} |\nabla w^+|^2 \varphi_R^2 + 2\int_{\Sigma_\lambda} |\nabla \varphi_R|^2 (w^+)^2\\
			&\qquad + C(\|u\|_{L^\infty(\Sigma_{\overline\lambda})}) \int_{\Sigma_\lambda} (w^+)^2 \varphi_R^2.
		\end{align*}
		This is equivalent to
		\begin{equation}\label{wcp1}
			\int_{\Sigma_\lambda} |\nabla w^+|^2 \varphi_R^2 \le 4 \int_{\Sigma_\lambda} |\nabla \varphi_R|^2 (w^+)^2 + 2C(\|u\|_{L^\infty(\Sigma_{\overline\lambda})}) \int_{\Sigma_\lambda} (w^+)^2 \varphi_R^2.
		\end{equation}

		By the classical Poincar\'e inequality in the interval $(0,\lambda)$, we have
		\begin{align*}
			\int_{\Sigma_\lambda} |\nabla w^+|^2 \varphi_R^2 &\ge \int_{\mathbb{R}^{N-1}} \left(\int_{0}^{\lambda} \left(\frac{\partial w^+}{\partial x_N}\right)^2 dx_N\right) \varphi_R^2(x') dx'\\
			&\ge \frac{\pi^2}{\lambda^2} \int_{\mathbb{R}^{N-1}} \left(\int_{0}^{\lambda} (w^+)^2 dx_N\right) \varphi_R^2(x') dx'\\
			&= \frac{\pi^2}{\lambda^2} \int_{\Sigma_\lambda} (w^+)^2 \varphi_R^2.
		\end{align*}

		Therefore, \eqref{wcp1} leads to
		\[
		\left(\frac{\pi^2}{\lambda^2} - 2C(\|u\|_{L^\infty(\Sigma_{\overline\lambda})})\right) \int_{\Sigma_\lambda} (w^+)^2 \varphi_R^2 \le 4 \int_{\Sigma_\lambda} |\nabla \varphi_R|^2 (w^+)^2.
		\]
		Choosing $\lambda_0=\pi[4+2C(\|u\|_{L^\infty(\Sigma_{\overline\lambda})})]^{-\frac{1}{2}}$, then for all $\lambda<\min\{\lambda_0, \overline\lambda\}$ we have
		\[
		\int_{\Sigma_\lambda} (w^+)^2 \varphi_R^2 \le \int_{\Sigma_\lambda} |\nabla \varphi_R|^2 (w^+)^2.
		\]

		Using \eqref{varphi}, we deduce
		\[
		\int_{\Sigma_\lambda^R} (w^+)^2 \le \frac{4}{R^2}\int_{\Sigma_\lambda^{2R}} (w^+)^2,
		\]
		where $\Sigma_\lambda^R := B'_R \times (0,\lambda)$. Setting $h(R):=\int_{\Sigma_\lambda^R} (w^+)^2$, we have
		\[
		h(R)\le\frac{4}{R^2}h(2R) \quad\text{ and }\quad h(R) \le C_\lambda R^{N-1} \quad\text{ for all } R>0.
		\]
		Hence $h(R)\le\frac{1}{2^N}h(2R)$ for all $R>2^\frac{N+2}{2}$. By iteration of this inequality, we obtain
		\[
		h(R)\le\frac{1}{2^{Nk}}h(2^kR)\le\frac{C_\lambda}{2^{Nk}}(2^kR)^{N-1}=\frac{C_\lambda}{2^k}R^{N-1}
		\]
		for all $k\in\mathbb{N}$ and $R>2^\frac{N+2}{2}$. Letting $k\to\infty$, we deduce $h(R)\equiv0$.
		This implies $\int_{\Sigma_\lambda} (w^+)^2=0$, which means $u\le v$ in $\Sigma_\lambda$ for all \add{$0<\lambda\le\lambda^*:=\frac12\min\{\lambda_0, \overline\lambda\}$}.
	\end{proof}

	\subsection{A priori estimates}
	We need the following property on solutions so that we can carry out the moving plane method to prove Theorem \ref{th:monotonicity}.
	\begin{lemma}\label{lem:uniform_limit}
		Assume that $f:(0,+\infty)\rev{\to\mathbb{R}}$ is a locally Lipschitz continuous function such that \eqref{F} holds. Let \( u \in C^2(\mathbb{R}^N_+) \cap C(\overline{\mathbb{R}^N_+}) \) be a solution to \eqref{main} with $u\in L^\infty(\Sigma_{\overline\lambda})$ for some $\overline\lambda>0$. Then
		\[
		\lim_{x_N\to0^+} u(x',x_N) = 0 \text{ uniformly in } x'\in\add{\mathbb{R}^{N-1}}.
		\]
	\end{lemma}

	\begin{proof}
		Let $\lambda^*<\overline{\lambda}$ be defined as in Proposition \ref{prop:wcp} and choose some $\rho>\|u\|_{L^\infty(\Sigma_{\lambda^*})}$.
		Let $h:(0,+\infty) \to \mathbb{R}$ be a $C^1$ function such that
		\begin{align*}
			h(t) > \max\{f(t), 0\} &\quad\text{ in } (0,\rho),\\
			h(t) = \frac{c}{t^2} &\quad\text{ in } [\rho,+\infty)
		\end{align*}
		for some $c>0$.
		We set $H(t) = \int_{\rho}^{t} h(s) ds$ for $t>0$, then $H$ is strictly increasing in $(0,+\infty)$ and $H(t) < \int_{\rho}^{+\infty} h(s) ds = \frac{c}{\rho}$. For each $\mu\ge\frac{c}{\rho}$, one can check that
		\[
		\int_{0}^{+\infty} \frac{ds}{\sqrt{\mu-H(s)}} > \int_{\rho}^{+\infty} \frac{ds}{\sqrt{\mu}} = +\infty
		\]
		and
		\[
		\int_{0}^{t} \frac{ds}{\sqrt{\mu-H(s)}} < \frac{t}{\sqrt{\mu-H(t)}} < +\infty \quad\text{ for } 0<t<+\infty.
		\]
		Hence the formula
		\[
		\int_{0}^{v_\mu(t)} \frac{ds}{\sqrt{\mu-H(s)}} = \sqrt{2} t \quad\text{ for all } t\ge0
		\]
		uniquely \add{determines} a function $v_\mu \in C^2(\mathbb{R}_+) \cap C(\overline{\mathbb{R}_+})$, which is a solution to the ODE problem
		\[
		\begin{cases}
			-v'' = h(v) &\text{ in } \mathbb{R}_+,\\
			v'(t) > 0 &\text{ in } \mathbb{R}_+,\\
			v(0) = 0.
		\end{cases}
		\]
		Moreover, $\lim_{\mu\to+\infty} v_\mu(t) = +\infty$ for all $t>0$.

		We fix some $\mu\add{\ge\frac{c}{\rho}}$ such that $v_\mu(\lambda^*) > \rho$. Then we choose $\lambda_0<\lambda^*$ satisfying $\|u\|_{L^\infty(\Sigma_{\lambda^*})} < v_\mu(\lambda_0) < \rho$. By abuse of notation, we will write $v_\mu(x',x_N):=v_\mu(x_N)$. Then $0<v_\mu<\rho$ in $\Sigma_{\lambda_0}$ and $u < v_\mu$ on $\{x_N=\lambda_0\}$.

		For small $\varepsilon>0$ such that $v_\mu(\lambda_0+\varepsilon)<\rho$, we define
		\[
		v_{\mu,\varepsilon}(x) := v_\mu(x + \varepsilon e_N).
		\]
		Then
		\[
		\begin{cases}
			v_\mu(\varepsilon) < v_{\mu,\varepsilon} < v_\mu(\lambda_0+\varepsilon) < \rho & \text{ in } \Sigma_{\lambda_0},\\
			-\Delta v_{\mu,\varepsilon} = h(v_{\mu,\varepsilon}) > f(v_{\mu,\varepsilon}) & \text{ in } \Sigma_{\lambda_0},\\
			u \le v_{\mu,\varepsilon} & \text{ on } \partial\Sigma_{\lambda_0}.
		\end{cases}
		\]
		Now Proposition \ref{prop:wcp} implies $u \le v_{\mu,\varepsilon}$ in $\Sigma_{\lambda_0}$. Letting $\varepsilon\to0$, we have $u \le v_{\mu}$ in $\Sigma_{\lambda_0}$ and the conclusion follows from \add{the} fact that $\lim_{t\to0^+} v_{\mu}(t) = 0$.
	\end{proof}

	We prove some a priori estimates for solutions to \eqref{main} in what follows.	The next lemma improves the upper bound on $u$ near the boundary in Lemma \ref{lem:uniform_limit} when an explicit upper bound on $f$ is given.
	\begin{lemma}\label{lem:upper_bound}
		Assume that $f:(0,+\infty)\rev{\to\mathbb{R}}$ is a locally Lipschitz continuous function such that \eqref{F} holds and $f(t) < \frac{c_0}{t^\gamma}$ for all $0<t<t_0$, where $c_0,t_0>0$ and $\gamma>1$. Let \( u \in C^2(\mathbb{R}^N_+) \cap C(\overline{\mathbb{R}^N_+}) \) be a solution to \eqref{main} with $u\in L^\infty(\Sigma_{\overline\lambda})$ for some $\overline\lambda>0$. Then
		\[
		u(x) \le C x_N^\frac{2}{\gamma+1} \quad\text{ in } \Sigma_{\lambda_0}
		\]
		for some constants $C, \lambda_0>0$.
	\end{lemma}
	\begin{proof}
		Let $\lambda^*<\overline{\lambda}$ be defined as in Proposition \ref{prop:wcp} and choose some $\rho>\|u\|_{L^\infty(\Sigma_{\lambda^*})}$. By compactness, there exists \add{$c_1\ge c_0$} such that $f(t) < \frac{c_1}{t^\gamma}$ for all $0<t<\rho$.
		Let
		\[
		v(t) \equiv \frac{(\gamma+1)^\frac{2}{\gamma+1}}{(2\gamma-2)^\frac{1}{\gamma+1}} t^\frac{2}{\gamma+1},
		\]
		then $v_\mu(x) = \mu v(x_N)$ solves $-\Delta v_\mu = \frac{\mu^{\gamma+1}}{v_\mu^\gamma}$ in $\mathbb{R}^N_+$.
		By abuse of notation, we will write $v_\mu(x_N):=v_\mu(x',x_N)$. We choose $\mu$ large such that $\mu^{\gamma+1}>c_1$ and $v_\mu(\lambda^*) > \rho$. Then we choose $\lambda_0<\lambda^*$ satisfying $\|u\|_{L^\infty(\Sigma_{\lambda^*})} < v_\mu(\lambda_0) < \rho$.  Now we have $0<v_\mu<\rho$ in $\Sigma_{\lambda_0}$ and $-\Delta v_\mu \ge f(v_\mu)$ in $\Sigma_{\lambda_0}$.

		For small $\varepsilon>0$ such that $v_\mu(\lambda_0+\varepsilon)<\rho$, we define
		\[
		v_{\mu,\varepsilon}(x) = v_\mu(x + \varepsilon e_N).
		\]
		Then
		\[
		\begin{cases}
			-\Delta v_{\mu,\varepsilon} \ge f(v_{\mu,\varepsilon}) & \text{ in } \Sigma_{\lambda_0},\\
			v_{\mu,\varepsilon} = v_\mu(\varepsilon) > 0 & \text{ on } \{x_N=0\},\\
			u < v_{\mu,\varepsilon} & \text{ on } \{x_N=\lambda_0\}.
		\end{cases}
		\]
		Now Proposition \ref{prop:wcp} implies $u \le v_{\mu,\varepsilon}$ in $\Sigma_{\lambda_0}$. Letting $\varepsilon\to0$ we conclude the proof.
	\end{proof}

	The next lemma concerns a lower bound on solutions.
	\begin{lemma}\label{lem:lower_bound}
		Assume that $f:(0,+\infty)\rev{\to\mathbb{R}}$ is a locally Lipschitz continuous function and $f(t) > \frac{c_0}{t^\gamma}$ for all $0<t<t_0$, where $c_0,t_0>0$ and $\gamma\ge0$. Let \( u \in C^2(\mathbb{R}^N_+) \cap C(\overline{\mathbb{R}^N_+}) \) be a solution to \eqref{main}. Then
		\[
		u(x) \ge \min\{C x_N^\frac{2}{\gamma+1}, t_0\} \quad\text{ in } \mathbb{R}^N_+
		\]
		for some constant $C>0$ independent of $t_0$.
	\end{lemma}
	\begin{proof}
		Let $\lambda_1>0$ and $\phi_1 \in C^2(\overline{B_1(0)})$ be the first eigenvalue and a corresponding positive eigenfunction of the Laplacian in $B_1(0)$, namely,
		\[
		\begin{cases}
			-\Delta \phi_1 = \lambda_1 \phi_1 & \text{ in } B_1(0),\\
			\phi_1 > 0 & \text{ in } B_1(0),\\
			\phi_1 = 0 & \text{ on } \partial B_1(0).
		\end{cases}
		\]

		\add{We set}
		\[
		w = \beta \phi_1^\frac{2}{\gamma+1},
		\]
		where $\beta>0$ will be chosen later. Direct calculation yields that
		\[
		-\Delta w = \frac{\alpha(x)}{w^\gamma} \quad \text{ in } B_1(0),
		\]
		where
		\[
		\alpha(x) := \frac{2 \beta^{\gamma+1}}{\gamma+1}\left(\frac{\gamma-1}{\gamma+1}|\nabla\phi_1(x)|^2 + \lambda_1\phi_1^2(x)\right).
		\]

		Now we fix $\beta>0$ such that $\sup_{x\in B_1(0)}\alpha(x) = c_0$ and hence
		\[
		-\Delta w \le \frac{c_0}{w^\gamma} \text{ in } B_1(0).
		\]

		Let $R_0>0$ be such that $R_0^\frac{2}{\gamma+1} w(0) = t_0$.
		For any $0<R\le R_0$ and $x_0 = (x_0', x_{0,N})\in\mathbb{R}^N$ with $x_{0,N} \ge R + \varepsilon$, where $\varepsilon$ is sufficiently small, we set
		\[
		w_{x_0,R}(x) := R^\frac{2}{\gamma+1}w\left(\frac{x-x_0}{R}\right) \quad\text{ in } B_R(x_0).
		\]
		Then
		\[
		w_{x_0,R} \le t_0 \quad\text{ and }\quad -\Delta w_{x_0,R} \le \frac{c_0}{w_{x_0,R}^\gamma} \quad\text{ in } B_R(x_0).
		\]

		On the other hand, since $w_{x_0,R}=0 < u$ on $\partial B_R(x_0)$, we can use $(w_{x_0,R}-u)^+ \chi_{B_R(x_0)}$ as a test function in
		\[
		-\Delta u = f(u)
		\quad\text{ and }\quad
		-\Delta w_{x_0,R} \le \frac{c_0}{w_{x_0,R}^\gamma}
		\]
		to obtain
		\[
		\int_{B_R(x_0)} \left| \nabla (w_{x_0,R}-u)^+ \right|^2
		\le \int_{B_R(x_0)} \left(\frac{c_0}{w_{x_0,R}^\gamma} - f(u)\right) (w_{x_0,R}-u)^+.
		\]
		In $B_R(x_0)\cap\{w_{x_0,R} > u\}$ we have $f(u) \ge \frac{c_0}{u^\gamma}$. Hence
		\[
		\int_{B_R(x_0)} \left| \nabla (w_{x_0,R}-u)^+ \right|^2
		\le \int_{B_R(x_0)} \left(\frac{c_0}{w_{x_0,R}^\gamma} - \frac{c_0}{u^\gamma}\right) (w_{x_0,R}-u)^+  \le 0.
		\]

		This implies $u \ge w_{x_0,R}$ in $B_R(x_0)$ with $x_{0,N} \ge R + \varepsilon$. Since $\varepsilon>0$ is arbitrary, we deduce
		\[
		u \ge w_{x_0,R} \text{ in } B_R(x_0) \text{ for all } 0<R\le R_0 \text{ and all } x_0 \in \mathbb{R}^N \text{ with } x_{0,N} \ge R.
		\]

		In particular, if $x_{0,N} = R < R_0$, then
		\[
		u(x_0) \ge w_{x_0,R}(x_0) = w(0)R^\frac{2}{\gamma+1} = w(0) x_{0,N}^\frac{2}{\gamma+1}.
		\]
		If $x_{0,N} \ge R = R_0$, then
		\[
		u(x_0) \ge w_{x_0,R}(x_0) = w(0)R_0^\frac{2}{\gamma+1} = t_0.
		\]

		The conclusion follows from the fact that $x_0$ is chosen arbitrarily in \add{$\mathbb{R}^N_+$}.
	\end{proof}

	Under a weaker assumption on $f$, we can still obtain a lower bound of $u$, which is useful in many situations.
	\begin{lemma}\label{lem:positive}
		Assume that $f:(0,+\infty)\rev{\to\mathbb{R}}$ is a locally Lipschitz continuous function and $f(t) > c_0t$ for all $0<t<t_0$, where $c_0,t_0>0$. Let \( u \in C^2(\mathbb{R}^N_+) \cap C(\overline{\mathbb{R}^N_+}) \) be a solution to \eqref{main}. Then
		\[
		u(x) \ge \min\{C x_N, \add{t_0/2}\} \quad\text{ in } \mathbb{R}^N_+
		\]
		for some constant $C>0$.
	\end{lemma}
	\add{A} similar result was obtained by Berestycki et al. \cite{MR1470317} for $C^2(\overline{\mathbb{R}^N_+})$ solutions and $f$ being locally Lipschitz continuous on $[0,+\infty)$. We provide a proof that works in our more general situation.
	\begin{proof}[Proof of Lemma \ref{lem:positive}]
		Let $\lambda_1>0$ and $\phi_1 \in C^2(\overline{B_1(0)})$ be the first eigenvalue and the corresponding positive eigenfunction of the Laplacian in $B_1(0)$ such that $\phi_1(0)=\add{t_0/2}$. We take $R=\sqrt{\frac{\lambda_1}{c_0}}$ and set $\phi_R(x) = \phi_1\left(\frac{x}{R}\right)$, then
		\[
		\begin{cases}
			-\Delta \phi_R = c_0 \phi_R \le f(\phi_R) & \text{ in } B_R(0),\\
			\phi_R > 0 & \text{ in } B_R(0),\\
			\phi_R = 0 & \text{ on } \partial B_R(0).
		\end{cases}
		\]
		Since $\phi_R$ is radially symmetric and by abuse of \add{notation}, we may write $\phi_R(x)=\phi_R(|x|)$.
		For each $x_0\in\mathbb{R}^N_+\setminus\Sigma_R$ we set
		\[
		\phi_R^{x_0}(x) = \phi_R(x - x_0) \quad\text{ for } x \in B_R(x_0).
		\]
		We will show that
		\begin{equation}\label{positive1}
			u \ge \phi_R^{x_0} \text{ in } B_R(x_0) \quad\text{ for every } x_0\in\mathbb{R}^N_+\setminus\Sigma_R.
		\end{equation}
		To this end, we let any $x_0:=(x_0', x_{0,N})\in\mathbb{R}^N_+\setminus\Sigma_R$.

		We only consider the case $x_{0,N}>R$ since the case $x_{0,N}=R$ can be derived by continuity. We define the set
		\[
		\Lambda := \{s\in(0,1] \mid u > t\phi_R^{x_0} \text{ in } B_R(x_0) \text{ for all } t\in(0,s) \}.
		\]
		Since $u$ is positive on \add{the} compact set $\overline{B_R(x_0)}$, we have $\Lambda \ne \emptyset$. We denote $s_0 = \sup\Lambda$. To derive \eqref{positive1}, we have to show that $s_0=1$. Assume by contradiction that $s_0<1$. We set $\tilde{\phi}:=s_0\phi_R^{x_0}$. Then
		\[
		u \ge \tilde{\phi} \quad\text{ in } B_R(x_0)
		\]
		\add{and}
		\begin{equation}\label{positive2}
			u(\hat x) = \tilde{\phi}(\hat x) \quad\text{ for some } \hat x \in B_R(x_0).
		\end{equation}
		From the boundary data of $\tilde{\phi}$, we can choose $\varepsilon>0$ small such that
		\begin{equation}\label{positive3}
			\hat x \in B_{R-\varepsilon}(x_0)\quad\text{ and }\quad u > \tilde{\phi} \quad\text{ on } \partial B_{R-\varepsilon}(x_0).
		\end{equation}
		Since $u$ and $\tilde{\phi}$ are positive in the set $\overline{B_{R-\varepsilon}(x_0)}$ and $f$ is locally Lipschitz continuous in $(0,+\infty)$, the strong comparison principle can be applied in $B_{R-\varepsilon}(x_0)$ to yield either $\tilde\phi < u$ in $B_{R-\varepsilon}(x_0)$ or $\tilde\phi = u$ in $B_{R-\varepsilon}(x_0)$. However, the former contradicts \eqref{positive2}, while the latter contradicts \eqref{positive3}.
		Hence \eqref{positive1} holds. This implies
		\[
		u(x) \ge \begin{cases}
			\phi_R(R-x_N) &\text{ if } x_N< R,\\
			\phi_R(0) &\text{ if } x_N\ge R.
		\end{cases}
		\]
		The conclusion follows immediately from the fact that $\phi_R'(R)<0$ and $\phi_R(0)=\add{t_0/2}$.
	\end{proof}

	\section{Qualitative properties of solutions}\label{sect3}

	\subsection{Monotonicity of solutions}
	As in the previous works, the main tool we use in proving the monotonicity of solutions is the method of moving planes, which was introduced by Alexandrov \cite{MR143162} in the context of differential geometry and by Serrin \cite{MR333220} in the PDE framework, for an overdetermined problem.
	We recall some familiar notions related to this method.
	For each $\lambda>0$, we denote
	\[
	x_\lambda = (x_1, x_2,\dots,2\lambda-x_N),
	\]
	which is the reflection of $x$ through the hyperplane $\partial\Sigma_\lambda$.
	Let $u$ be a solution to \eqref{main}. We set
	\[
	u_\lambda(x) = u(x_\lambda),
	\]
	then $u_\lambda$ satisfies $-\Delta u_\lambda = f(u_\lambda)$ in $\Sigma_{2\lambda}$.

	We are ready to prove the main results in this section.
	\begin{proof}[Proof of Theorem \ref{th:monotonicity}]
		Applying Proposition \ref{prop:wcp} with $v\equiv u_\lambda$, we find $\lambda^*>0$ such that $u\le u_\lambda$ in $\Sigma_\lambda$ for all $\lambda\le\lambda^*$. Hence the set
		\[
		\Lambda = \{\lambda\in(0,+\infty) \mid u\le u_\mu \text{ in } \Sigma_\mu \text{ for all } \mu\in(0,\lambda]\}
		\]
		is not empty.
		Therefore, we can define
		\[
		\lambda_0 = \sup \Lambda.
		\]
		We show that $\lambda_0=+\infty$.

		By contradiction, we assume that $\lambda_0<+\infty$. By continuity, we know that
		\begin{equation}\label{monotonicity1}
			u \le u_{\lambda_0} \quad\text{ in } \Sigma_{\lambda_0}.
		\end{equation}
		By Lemmas \ref{lem:uniform_limit} and \ref{lem:positive}, there exist $\tilde\lambda,\tilde\delta>0$ sufficiently small such that
		\begin{equation}\label{monotonicity2}
			u + \tilde\delta < u_\lambda \text{ in } \Sigma_{\tilde\lambda} \quad\text{ for all } \lambda\ge\lambda_0.
		\end{equation}
		We will reach a contradiction by showing that for some small $\varepsilon_0>0$,
		\[
		u \le u_\lambda \text{ in } \Sigma_\lambda \quad\text{ for all } \lambda\in[\lambda_0,\lambda_0+\varepsilon_0].
		\]

		If this is not true, then there exist $\lambda_n\searrow\lambda_0$ and $x_n:=(x_n',x_{n,N})\in\Sigma_{\lambda_n} \setminus \Sigma_{\tilde\lambda}$ such that
		\begin{equation}\label{monotonicity3}
			u(x_n) > u_{\lambda_n}(x_n).
		\end{equation}
		Up to a subsequence, we may assume that $x_{n,N}\to y_0\in[\tilde\lambda,\lambda_0]$ as $n\to\infty$.
		Now we set
		\[
		u_n(x',x_N) := u(x'+x_n',x_N).
		\]

		By Lemma \ref{lem:positive}, we know that
		\[
		\min\{C \tilde\lambda, \add{t_0/2}\} \le u_n \le \|u\|_{L^\infty(\Sigma_{\lambda})} \quad\text{ in } \Sigma_\lambda\setminus \Sigma_{\tilde\lambda} \text{ for } \lambda>0.
		\]
		Hence $f(u_n)$ are also bounded on each strip $\Sigma_\lambda\setminus \Sigma_{\tilde\lambda}$.
		By standard regularity, Ascoli-Arzel\`{a}'s theorem and a diagonal process, we deduce that
		\[
		u_n \to v \quad\text{ in } C^2_{\rm loc}(\mathbb{R}^N_+\setminus \Sigma_{\tilde\lambda})
		\]
		up to a subsequence, where $v$ \add{solves $-\Delta v = f(v)$} in $\mathbb{R}^N_+\setminus \overline{\Sigma_{\tilde\lambda}}$.
		Moreover, \eqref{monotonicity1}, \eqref{monotonicity3} imply $v \le v_{\lambda_0}$ in $\Sigma_{\lambda_0}\setminus \Sigma_{\tilde\lambda}$ and $v(0',y_0) \ge v_{\lambda_0}(0',y_0)$. Hence
		\[
		v(0',y_0) = v_{\lambda_0}(0',y_0).
		\]

		Therefore,
		\begin{equation}\label{monotonicity4}
			-\Delta(v_{\lambda_0}-v) + C(v_{\lambda_0}-v) = f(v_{\lambda_0})-f(v) + C(v_{\lambda_0}-v) \ge 0
		\end{equation}
		in any compact set of $\{\tilde\lambda<x_N\le\lambda_0\}$ with sufficiently large $C$.
		By the strong maximum principle, we deduce $v < v_{\lambda_0}$ in $\Sigma_{\lambda_0}\setminus \overline{\Sigma_{\tilde\lambda}}$. (The case $v \equiv v_{\lambda_0}$ in $\Sigma_{\lambda_0}$ cannot happen due to $v < v_{\lambda_0}$ on $\{x_N=\tilde\lambda\}$ deduced from \eqref{monotonicity2}.) This implies $y_0=\lambda_0$. By the mean value theorem, there exists $\xi_n\in(x_{n,N},2\lambda_n-x_{n,N})$ such that
		\[
		\frac{\partial u_n}{\partial x_N}(0',\xi_n) = \frac{u_n(0',2\lambda_n-x_{n,N}) - u_n(0',x_{n,N})}{2\lambda_n-2x_{n,N}} = \frac{u_{\lambda_n}(x_n) - u(x_n)}{2\lambda_n-2x_{n,N}} < 0.
		\]

		Letting $n\to\infty$, we obtain
		\[
		\frac{\partial v}{\partial x_N}(0',\lambda_0) \le 0.
		\]
		Hence
		\[
		\frac{\partial (v_{\lambda_0} - v)}{\partial x_N}(0',\lambda_0) = -2\frac{\partial v}{\partial x_N}(0',\lambda_0) \ge 0.
		\]
		However, this contradicts the \add{Hopf} lemma \cite{MR1814364} for \eqref{monotonicity4} in $\Sigma_{\lambda_0}\setminus \Sigma_{\tilde\lambda}$.

		Therefore, $\lambda_0=+\infty$. Hence $u \le u_\lambda$ in $\Sigma_\lambda$ for all $\lambda>0$. Exploiting the strong maximum principle and the \add{Hopf} lemma for $u_\lambda - u$ as above we deduce
		\[
		\frac{\partial u}{\partial x_N} > 0 \quad\text{ in } \add{\mathbb{R}^N_+},
		\]
		which is what we have to prove.
	\end{proof}

	\begin{proof}[Proof of Theorem \ref{th:monotonicity2}]
		Since $g:[0,+\infty)\rev{\to\mathbb{R}}$ is locally Lipschitz continuous, there exist $t_0,\add{k_1,k_2}>0$ such that the function $f(t) = \frac{1}{t^\gamma} + g(\add{t})$ is decreasing on $(0,t_0)$ and
		\[
		\frac{\add{k_1}}{t^\gamma}<f(t)<\frac{\add{k_2}}{t^\gamma}\quad\text{ in } (0,t_0).
		\]
		\add{Moreover, $f$ satisfies \eqref{F}, since $t\mapsto t^{-\gamma}$ is decreasing and $g$ is Lipschitz continuous on $[0,M]$ for every $M>0$.}
		Hence Lemmas \ref{lem:upper_bound} and \ref{lem:lower_bound} imply the existence of $\lambda_0>0$ such that
		\begin{equation}\label{bounds}
			c x_N^\frac{2}{\gamma+1} \le u(x) \le C x_N^\frac{2}{\gamma+1} \quad\text{ in } \Sigma_{\lambda_0}.
		\end{equation}

		\add{If $u\in L^\infty(\Sigma_\lambda)$ for all $\lambda>0$, then the monotonicity of $u$ in $\mathbb{R}^N_+$ follows from Theorem \ref{th:monotonicity}, because $f(t)\to+\infty$ as $t\to0^+$. The positivity of $\frac{\partial u}{\partial x_N}$ in $\Sigma_{\lambda_0}$ is the case $\eta=e_N$ of \eqref{gradient_blowup}.} So we only prove \rev{\eqref{gradient_upper} and} \eqref{gradient_blowup}. Our proof is motivated by an idea from \cite{MR3459013}.

		Let $(\varepsilon_n)$ be a positive sequence such that $\varepsilon_n\to0$ as $n\to\infty$. We define
		\[
		w_n(x) := \varepsilon_n^{-\frac{2}{\gamma+1}} u(\varepsilon_n x) \quad\text{ for } x \in \mathbb{R}^N_+.
		\]
		\add{By \eqref{bounds}, we have}
		\begin{equation}\label{wn_bound}
			\add{c\, x_N^\frac{2}{\gamma+1} \le w_n(x) \le C x_N^\frac{2}{\gamma+1} \quad\text{ in } \Sigma_{\lambda_0/\varepsilon_n}.}
		\end{equation}
		Moreover, $w_n$ solves
		\begin{equation}\label{wn}
			-\Delta w_n = \frac{1}{w_n^\gamma} + \varepsilon_n^\frac{2\gamma}{\gamma+1} g(\varepsilon_n^\frac{2}{\gamma+1} w_n) \quad \text{ in } \mathbb{R}^N_+.
		\end{equation}
		\add{Let $K\subset\mathbb{R}^N_+$ be compact and let $K'\subset\mathbb{R}^N_+$ be a compact neighborhood of $K$. For $n$ large, $K'\subset\Sigma_{\lambda_0/\varepsilon_n}$, so that $w_n$ is bounded from above and from below by positive constants in $K'$ by \eqref{wn_bound}, and $0<\varepsilon_n^\frac{2}{\gamma+1}w_n\le C\lambda_0^\frac{2}{\gamma+1}$ there. Hence the right-hand side of \eqref{wn} is uniformly bounded in $K'$, and standard interior estimates \cite{MR1814364} (first in $W^{2,p}$, then Schauder estimates) show that $(w_n)$ is bounded in $C^{2,\alpha}(K)$, uniformly for $n$ large, for some $0<\alpha<1$. These bounds depend only on $K$, $K'$ and the constants in \eqref{bounds}; in particular, they remain valid if $w_n$ is replaced by its horizontal translates $w_n(\cdot+y',\cdot)$, $y'\in\mathbb{R}^{N-1}$.}

		{\color{black}
		We first prove \eqref{gradient_upper}. Let $x_0=(x_0',x_{0,N})\in\Sigma_{\lambda_0/4}$ and set $\varepsilon:=x_{0,N}$. The function
		\[
		w(x) := \varepsilon^{-\frac{2}{\gamma+1}} u(x_0'+\varepsilon x', \varepsilon x_N), \qquad x\in\mathbb{R}^N_+,
		\]
		satisfies, by \eqref{bounds},
		\[
		c\, 2^{-\frac{2}{\gamma+1}} \le w \le C\, 2^\frac{2}{\gamma+1} \quad\text{ in } B_{\frac12}(e_N),
		\]
		since $B_{\frac12}(e_N)\subset\Sigma_2\setminus\Sigma_{\frac12}$ and $2\varepsilon<\lambda_0$. Moreover, $w$ solves \eqref{wn} with $\varepsilon_n$ replaced by $\varepsilon$. Since $0<\varepsilon^\frac{2}{\gamma+1}w\le C(2\lambda_0)^\frac{2}{\gamma+1}$ in $B_{\frac12}(e_N)$ and $g$ is bounded on bounded subsets of $[0,+\infty)$, we have $|\Delta w|\le K$ in $B_{\frac12}(e_N)$, where $K$ does not depend on $x_0$. The interior gradient estimate \cite[Theorem 3.9]{MR1814364} gives $|\nabla w(e_N)|\le K'$ for some $K'$ independent of $x_0$. Since
		\[
		\nabla w(e_N) = \varepsilon^\frac{\gamma-1}{\gamma+1}\nabla u(x_0) = x_{0,N}^\frac{\gamma-1}{\gamma+1}\nabla u(x_0),
		\]
		we obtain \eqref{gradient_upper} with $\lambda_1=\lambda_0/4$ and $c_2=2K'$. Notice that $c_2$ and $\lambda_1$ do not depend on any direction.
		}

		Now we prove the estimate from below. Suppose by contradiction that there exist $\beta>0$, a sequence of \add{unit} vectors $\eta_n\in \mathbb{S}^{N-1}_+$ with $(\eta_n,e_N) \ge \beta$ and a sequence of points $x_n=(x_n', x_{n,N})\in\mathbb{R}^N_+$ such that
		\begin{equation}\label{cassump}
			x_{n,N}^\frac{\gamma-1}{\gamma+1} \frac{\partial u(x_n)}{\partial \eta_n} \to 0 \text{ and } x_{n,N} \to 0 \quad\text{ as } n\to\infty.
		\end{equation}
		Passing to a subsequence, we may assume $\eta_n\to\eta\in \mathbb{S}^{N-1}_+$ with $(\eta,e_N) \ge \beta$ as $n\to\infty$.
		We define $w_n$ as above with $\varepsilon_n=x_{n,N}$ and $\tilde w_n(x',x_N)=w_n(x'+\varepsilon_n^{-1}x'_n,x_N)$, namely,
		\[
		\tilde w_n(x) := x_{n,N}^{-\frac{2}{\gamma+1}} u(x_{n,N} x'+x_n', x_{n,N} x_N) \quad\text{ for } x \in \mathbb{R}^N_+.
		\]
		Then \eqref{wn_bound} and \eqref{wn} still hold for $\tilde w_n$\add{, and $(\tilde w_n)$ is bounded in $C^{2,\alpha}(K)$ for every compact set $K\subset\mathbb{R}^N_+$, uniformly for $n$ large. By the Arzel\`a--Ascoli theorem and a diagonal argument along an exhaustion of $\mathbb{R}^N_+$ by compact sets, there exist a subsequence, still denoted by $(\tilde w_n)$, and a function $w_\infty\in C^2(\mathbb{R}^N_+)$ such that
		\[
		\tilde w_n \to w_\infty \quad\text{ in } C^2_{\rm loc}(\mathbb{R}^N_+).
		\]
		Since $\varepsilon_n^\frac{2\gamma}{\gamma+1}\to0$ and $g$ is bounded on bounded sets, passing to the limit in \eqref{wn} and \eqref{wn_bound} yields
		\[
		-\Delta w_\infty = \frac{1}{w_\infty^\gamma} \quad\text{ and }\quad c\, x_N^\frac{2}{\gamma+1}\le w_\infty(x) \le C x_N^\frac{2}{\gamma+1} \quad\text{ in } \mathbb{R}^N_+.
		\]}
		Hence $w_\infty$ \add{extends continuously to $\overline{\mathbb{R}^N_+}$ with $w_\infty=0$ on $\partial\mathbb{R}^N_+$, it} is a solution to \eqref{pure}\add{, and it is bounded in every strip $\Sigma_\lambda$}.
		By \cite[Theorem 1]{MR4753083}, $w_\infty$ depends only on $x_N$ and $w_\infty'>0$ in $\mathbb{R}_+$.

		On the other hand, \eqref{cassump} gives $\frac{\partial \tilde w_n(e_N)}{\partial \eta_n} = x_{n,N}^\frac{\gamma-1}{\gamma+1} \frac{\partial u(x_n)}{\partial \eta_n} \to 0$ as $n \to \infty$. This is a contradiction since $\frac{\partial \tilde w_n(e_N)}{\partial \eta_n} \to \frac{\partial w_\infty(e_N)}{\partial \eta}=w_\infty'(1)\eta_N>0$.
	\end{proof}

	\subsection{Rigidity of solutions}
	In this subsection, we prove Theorem \ref{th:rigidity}.
	We will make use of the following version of the maximum principle in unbounded domains which is due to Berestycki, Caffarelli and Nirenberg.
	\begin{lemma}[Lemma 2.1 in \cite{MR1470317}]\label{lem:wmp}
		Let $D$ be a domain (open connected set) in $\mathbb{R}^N$, possibly unbounded.
		Assume that $\overline D$ is disjoint from the closure of an infinite open connected
		cone $\Sigma$. Suppose there is a function $w \in C^2(D) \cap C(\overline D)$ that is bounded above
		and satisfies for some continuous function $c(x)$,
		\begin{align*}
			-\Delta w - c(x) w \le 0 &\text{ in } D \text{ with } c(x) \le 0,\\
			w \le 0 &\text{ on } \partial D.
		\end{align*}
		Then $w \le 0$ in $D$.
	\end{lemma}

	Lemma \ref{lem:wmp} allows us to derive a weak comparison principle.
	Notice that in the following result, we do not assume that $v$ is bounded from above.
	\begin{proposition}\label{prop:wcp2}
		Let $f:(0,+\infty)\rev{\to\mathbb{R}}$ be a locally Lipschitz continuous function which is non-increasing on $(t_1,+\infty)$ for some $t_1>0$.
		Let \( u, v\in C^2(\Omega) \cap C(\overline{\Omega}) \)  be solutions to
		\[
		\begin{cases}
			-\Delta w = f(w) &\text{ in } \Omega,\\
			w>0 &\text{ in } \Omega,
		\end{cases}
		\]
		where $\Omega\subset\mathbb{R}^N$ is an open connected set such that $\mathbb{R}^N\setminus\overline\Omega$ contains an infinite open connected cone \add{whose closure is disjoint from $\overline\Omega$}.
		Assume that
		\[
		u \le v \text{ on } \partial\Omega, \quad v(x) \ge t_1 \text{ in } \Omega
		\]
		and
		\[
		\sup_\Omega (u-v) < +\infty.
		\]
		Then $u \le v$ in $\Omega$.
	\end{proposition}

	\begin{proof}
		Assume by contradiction that $u > v$ somewhere in $\Omega$. Let $D$ be a connected component of the set where $u > v$. Setting
		\[
		w=u-v,
		\]
		then
		\[
		\begin{cases}
			-\Delta w - c(x) w = 0 &\text{ in } D,\\
			w = 0 &\text{ on } \partial D,
		\end{cases}
		\]
		where
		\[
		c(x) = \frac{f(u(x))-f(v(x))}{u(x)-v(x)}.
		\]

		Moreover, since $t_1 \le v < u$ in $D$ and $f$ is non-increasing on $(t_1,+\infty)$, we deduce $c(x)\le0$ in $D$. Hence Lemma \ref{lem:wmp} applies to yield $w \le 0$ in $D$, a contradiction. Therefore, $u \le v$ in $\Omega$.
	\end{proof}

	We employ the technique from \cite[Proposition 5]{MR4753083} to show that solutions to problem \eqref{main} grow at most at a linear rate as $x_N \to +\infty$.
	\begin{lemma}\label{lem:upper_bound_infty}
		Under the assumptions of Theorem \ref{th:rigidity}, there exists a constant $C>0$ such that
		\[
		u(x) \le C x_N \quad\text{ in } \mathbb{R}^N_+\setminus\Sigma_{\overline{\lambda}}.
		\]
	\end{lemma}
	\begin{proof}
		If $u$ is a solution to \eqref{main}, then
		\[
		v(x) := \left(\frac{\overline\lambda}{2}\right)^{-\frac{2}{\gamma+1}} u\left(\frac{\overline\lambda}{2}x\right)
		\]
		is bounded in $\Sigma_2$ and $v$ satisfies
		\[
		-\Delta v = \left(\frac{\overline\lambda}{2}\right)^\frac{2\gamma}{\gamma+1} f\left(\left(\frac{\overline\lambda}{2}\right)^\frac{2}{\gamma+1} v\right).
		\]
		Moreover, the function $\add{\tilde f}(t) := \left(\frac{\overline\lambda}{2}\right)^\frac{2\gamma}{\gamma+1} f\left(\left(\frac{\overline\lambda}{2}\right)^\frac{2}{\gamma+1} t\right)$ still satisfies \add{\eqref{F},} (i) and (ii) in Theorem \ref{th:rigidity} with \add{possibly} different parameters $c_1,t_1$.
		Therefore, without loss of generality, we may assume that our solution $u$ is bounded in the strip $\Sigma_2$.

		From (i) and  Lemma \ref{lem:lower_bound}, we have
		\begin{equation}\label{s1}
			u(x) \ge \tilde{C} x_N^\frac{2}{\gamma+1} \quad\text{ in } \mathbb{R}^N_+.
		\end{equation}
		From (ii) and the compactness, there exists $c_2>0$ such that
		\begin{equation}\label{s2}
			f(t) < \frac{c_2}{t^\gamma} \quad\text{ for all } t\ge \tilde{C}.
		\end{equation}

		Let any $x_0 = (x_0', x_{0,N}) \in \mathbb{R}^N_+$ with $x_{0,N}:=4R>2$. We set
		\[
		u_R(x) := R^{-\frac{2}{\gamma+1}} u(x_0+R(x-x_0)),
		\]
		then $u_R>0$ in $B_4(x_0)$ and
		\[
		-\Delta u_R(x) = R^\frac{2\gamma}{\gamma+1} f(u(x_0+R(x-x_0))) = \add{\frac{R^2 f(u(x_0+R(x-x_0)))}{u(x_0+R(x-x_0))}} u_R(x).
		\]

		Remark that $x_0+R(x-x_0) \in \mathbb{R}^N_+\setminus\Sigma_{2R} \subset \mathbb{R}^N_+\setminus\Sigma_1$ for $x\in B_2(x_0)$.
		Hence \eqref{s1} and \eqref{s2} give
		\[
		\add{\frac{R^2 f(u(x_0+R(x-x_0)))}{u(x_0+R(x-x_0))}} \le \frac{c_2 R^2}{u(x_0+R(x-x_0))^{\gamma+1}} \le \frac{c_2}{4 \tilde{C}^{\gamma+1}} \quad\text{ in } B_2(x_0).
		\]
		By Harnack's inequality, we have
		\[
		\sup_{B_1(x_0)} u_R \le C_H \inf_{B_1(x_0)} u_R,
		\]
		which implies
		\[
		\sup_{B_R(x_0)} u \le C_H \inf_{B_R(x_0)} u,
		\]
		where $C_H>0$ is independent of $x_0$. From this point, we can proceed as in the proof of \cite[Proposition 5]{MR4753083} to get the thesis.
	\end{proof}

	Given the previous asymptotic bound on \( u \), we can apply the scaling technique as in \cite[Proposition 7]{MR4753083} to establish a bound on the gradient.
	\begin{lemma}\label{lem:gradient_bound}
		Under the assumptions of Theorem \ref{th:rigidity}, there exists a constant $C>0$ such that
		\[
		|\nabla u(x)| \le C \quad\text{ in } \mathbb{R}^N_+\setminus\Sigma_{\overline{\lambda}}.
		\]
	\end{lemma}
	\begin{proof}
		As in Lemma \ref{lem:upper_bound_infty}, we may assume $\overline{\lambda}=2$.
		Let $x_0\in\mathbb{R}^N_+\setminus\Sigma_2$ and set $R=x_{0,N}\ge2$. We define
		\[
		u_R(x) := \frac{u(Rx)}{R} \quad\text{ in } B_{\frac{1}{2}}\left(\frac{x_0}{R}\right).
		\]
		By Lemma \ref{lem:upper_bound_infty} we have $u_R \le C$. Moreover, from \eqref{s1} and \eqref{s2}, we deduce
		\[
		-\Delta u_R = Rf(u(Rx)) \le \frac{c_2 R}{u(Rx)^\gamma} \le \frac{4^\frac{\gamma}{\gamma+1} c_2}{\tilde{C}^\gamma} R^{-\frac{\gamma-1}{\gamma+1}} \le \frac{2 c_2}{\tilde{C}^\gamma} \quad\text{ in } B_{\frac{1}{2}}\left(\frac{x_0}{R}\right).
		\]
		By the standard gradient estimate, we have $|\nabla u_R| \le C'$ in $B_{\frac{1}{4}}\left(\frac{x_0}{R}\right)$. This indicates $|\nabla u| \le C'$ in $B_{\frac{R}{4}}(x_0)$. The thesis follows from the arbitrariness of $x_0$.
	\end{proof}

	We are ready to prove Theorem \ref{th:rigidity} by employing the sliding method.

	\begin{proof}[Proof of Theorem \ref{th:rigidity}]
		For each $\lambda>0$ and $\nu\in \mathbb{S}^{N-1}_+$, we define
		\[
		u_\lambda^\nu (x) := u(x+\lambda\nu).
		\]

		We aim to show that
		\begin{equation}\label{nu_monotone}
			u \le u_\lambda^\nu \,\text{ in } \mathbb{R}^N_+ \quad\text{ for all } \lambda>0.
		\end{equation}
		By Lemma \ref{lem:lower_bound}, there exists $\lambda^*>0$ such that
		\[
		u(x)>t_1 \text{ for } x_N>\lambda^*.
		\]
		Let $\lambda>\lambda_\nu^*$, where $\lambda_\nu^*:=\frac{\lambda^*}{(\nu,e_N)}$, then $u_\lambda^\nu>t_1$ in $\mathbb{R}^N_+$. Moreover, from Lemmas \ref{lem:upper_bound_infty}, \ref{lem:gradient_bound} and the mean value theorem, we deduce
		\[
		\sup_{\mathbb{R}^N_+} (u-u_\lambda^\nu) < +\infty.
		\]
		Since
		\[
		\begin{cases}
			-\Delta u = f(u) &\text{ in } \mathbb{R}^N_+,\\
			-\Delta u_\lambda^\nu = f(u_\lambda^\nu) &\text{ in } \mathbb{R}^N_+,\\
			u>0 &\text{ in } \mathbb{R}^N_+,\\
			u_\lambda^\nu>t_1 &\text{ in } \mathbb{R}^N_+,\\
			u \le u_\lambda^\nu &\text{ on } \partial\mathbb{R}^N_+,
		\end{cases}
		\]
		we can apply Proposition \ref{prop:wcp2} to derive
		\[
		u \le u_\lambda^\nu \,\text{ in } \mathbb{R}^N_+ \quad\text{ for all } \lambda>\lambda_\nu^*.
		\]

		Now that the set
		\[
		\Lambda = \{ \lambda>0 \mid u \le u_\mu^\nu \text{ in } \mathbb{R}^N_+ \text{ for all } \mu > \lambda \}
		\]
		is nonempty, we can define $\lambda_0 = \inf \Lambda$.
		We will show that
		\[
		\lambda_0 = 0.
		\]

		\add{Assume to the contrary} that $\lambda_0>0$. By continuity of $u$, we have $u \le u_{\lambda_0}^\nu$ in $\mathbb{R}^N_+$. In order to reach a contradiction, we will search for some $\varepsilon_0$ small such that
		\begin{equation}\label{contradiction}
			u \le u_\lambda^\nu \quad\text{ in } \mathbb{R}^N_+
		\end{equation}
		for all $\lambda \in (\lambda_0-\varepsilon_0, \lambda_0]$.

		$\circ$ Due to Lemmas \ref{lem:uniform_limit} and \ref{lem:lower_bound}, there exist $\tilde{\lambda},\tilde{\delta}>0$ sufficiently small such that
		\begin{equation}\label{small_strip}
			u+\tilde{\delta}\leq u_\lambda^\nu \quad\text{ in } \Sigma_{\tilde{\lambda}}
		\end{equation}
		for all $\lambda>\lambda_0/2$.

		$\circ$ We claim that
		\begin{equation}\label{middle_strip}
			u \le u_\lambda^\nu \quad\text{ in } \Sigma_{\lambda^*}\setminus\Sigma_{\tilde{\lambda}}
		\end{equation}
		for all $\lambda \in (\lambda_0-\varepsilon_0, \lambda_0)$, where $\varepsilon_0>0$ is sufficiently small.

		Assume that \eqref{middle_strip} does not hold. Then there exist two sequences $\lambda_n \nearrow \lambda_0$ and $x_n := (x'_n, x_{n,N}) \in \mathbb{R}^{N-1} \times [\tilde{\lambda}, \lambda^*)$ such that
		\begin{equation}\label{contra_sequence}
			u(x_n) > u_{\lambda_n}^\nu(x_n).
		\end{equation}
		Moreover, we may assume $x_{n,N} \to y_0 \in [\tilde{\lambda}, \lambda^*]$. Now we set
		\[
		u_n(x', x_N) = u(x'+x'_n, x_N).
		\]

		Since $C \tilde{\lambda}^\frac{2}{\gamma+1} \le u_n \le \|u\|_{L^\infty(\Sigma_\lambda)}$ in $\Sigma_\lambda\setminus\Sigma_{\tilde{\lambda}}$, we have that $f(u_n)$ is bounded in $\Sigma_\lambda\setminus\Sigma_{\tilde{\lambda}}$ for each $\lambda>\tilde{\lambda}$. The standard regularity gives $\|u_n\|_{C^{2,\alpha}(\overline{\Sigma_\lambda\setminus\Sigma_{\tilde{\lambda}}})} < C_\lambda$ for some $0<\alpha<1$. By the Arzel\`a--Ascoli theorem, via a standard diagonal process, we have
		\[
		u_n \to v \quad\text{ in } C^2_{\rm loc}(\overline{\mathbb{R}^N_+\setminus\Sigma_{\tilde{\lambda}}})
		\]
		up to a subsequence. Moreover, $v$ \add{solves} $-\Delta v = f(v)$ in $\mathbb{R}^N_+\setminus\overline{\Sigma_{\tilde{\lambda}}}$.
		Using the definition of $\lambda_0$ and passing \eqref{contra_sequence} to the limit, we have
		\begin{align*}
			&v \le v_{\lambda_0}^\nu \quad\text{ in } \mathbb{R}^N_+\setminus\Sigma_{\tilde{\lambda}},\\
			&v(x_0) = v_{\lambda_0}^\nu(x_0),
		\end{align*}
		where $x_0 = (0',y_0)$. On the other hand, by \eqref{small_strip} we have $v + \tilde{\delta} \le v_{\lambda_0}^\nu$ on $\{x_N=\tilde{\lambda}\}$. Hence the strong comparison principle implies $v < v_{\lambda_0}^\nu$ in $\mathbb{R}^N_+\setminus\Sigma_{\tilde{\lambda}}$. This contradicts the fact that $v(x_0) = v_{\lambda_0}^\nu(x_0)$.
		Therefore, \eqref{middle_strip} must hold.

		$\circ$ Next, we show that
		\begin{equation}\label{outer_strip}
			u \le u_\lambda^\nu \quad\text{ in } \mathbb{R}^N_+\setminus\Sigma_{\lambda^*}
		\end{equation}
		for all $\lambda \in (\lambda_0-\varepsilon_0, \lambda_0)$.

		From \eqref{middle_strip} and the continuity, we already have $u \le u_\lambda^\nu$ on $\{x_N=\lambda^*\}$. Moreover, $u_\lambda^\nu(x) \ge t_1$ for each $x\in\mathbb{R}^N_+\setminus\Sigma_{\lambda^*}$. Hence \eqref{outer_strip} follows by applying Proposition \ref{prop:wcp2} with $u$ and $v:=u_\lambda^\nu$ on $\mathbb{R}^N_+\setminus\Sigma_{\lambda^*}$.

		Combining \eqref{small_strip}, \eqref{middle_strip} and \eqref{outer_strip}, we obtain \eqref{contradiction}. This contradicts the definition of $\lambda_0$ and hence \eqref{nu_monotone} is proved.

		Therefore, $u$ is monotone increasing in direction $\nu$ for all $\nu\in \mathbb{S}^{N-1}_+$. That is,
		\[
		\frac{\partial u}{\partial \nu} := (\nabla u, \nu) \ge 0 \quad \text{ in } \mathbb{R}^N_+.
		\]
		To deduce the one-dimensional symmetry of $u$, we take $\zeta$ \add{to be} any \add{unit vector with $\zeta_N=0$. Let} $\nu_n\in \mathbb{S}^{N-1}_+$ be a sequence converging to $\zeta$. \add{Then} $\frac{\partial u}{\partial \nu_n} \ge 0$\add{, and letting} $n\to\infty$\add{ gives} $\frac{\partial u}{\partial\zeta} \ge 0$ in $\mathbb{R}^N_+$.
		\add{Applying the same argument to a sequence in $\mathbb{S}^{N-1}_+$ converging to $-\zeta$}, we obtain $\frac{\partial u}{\partial\zeta} \le 0$ in $\mathbb{R}^N_+$.
		Therefore, $u$ is constant in direction $\zeta$. Since $\zeta$ is arbitrary, we deduce that $u$ does not depend on $x'$. Hence $u$ depends only on $x_N$ and \add{is nondecreasing} in $x_N$.

		By writing $v(x_N)=u(x)$, \add{we have $v'\ge0$ and $v''=-f(v)<0$ in $\mathbb{R}_+$. Hence $v'$ is strictly decreasing, and if $v'(t_0)=0$ for some $t_0>0$, then $v'<0$ in $(t_0,+\infty)$, which is impossible. Therefore $v'>0$ in $\mathbb{R}_+$, and} problem \eqref{main} reduces to
		\[
		\begin{cases}
			-v'' = f(v) &\text{ in } \mathbb{R}_+,\\
			v'(t) > 0 &\text{ in } \mathbb{R}_+,\\
			v(0) = 0.
		\end{cases}
		\]
		Hence for every $t>0$, we have
		\begin{equation}\label{ode1}
			\frac{1}{2} (v')^2 - F(v) = M,
		\end{equation}
		which is a constant. Letting $t\to+\infty$ and noticing that $v(t)\to+\infty$ by Lemma \ref{lem:lower_bound}, we deduce $M\ge0$.
		By integrating \eqref{ode1} and using $v(0)=0$, this gives
		\begin{equation}\label{ode2}
			\int_{0}^{v(t)} \frac{ds}{\sqrt{M+F(s)}} = \sqrt{2} t \quad\text{ for all } t\ge0.
		\end{equation}

		Conversely, for every $M\ge0$ we have
		\[
		\int_{0}^{+\infty} \frac{ds}{\sqrt{M+F(s)}} = +\infty
		\text{ and }
		\int_{0}^{t} \frac{ds}{\sqrt{M+F(s)}} < +\infty \text{ for all } t>0.
		\]
		Therefore, for each $M\ge0$, formula \eqref{ode2} uniquely determines a function $v:=v_M$ which is a solution to \eqref{main}. It is also clear from \eqref{ode1} that each solution $v_M$ is characterized by the property $\lim_{t\to+\infty} v_M'(t) =
		\sqrt{2M}$.
	\end{proof}

	Finally, we discuss the case that $f$ is nonincreasing in the whole domain $(0,+\infty)$.
	\begin{proof}[Proof of Theorem \ref{th:rigidity2}]
		\add{The one-dimensional symmetry conclusion already follows from \cite[Theorem~III]{PV04}. We include the following direct Kelvin-transform argument for completeness, in the spirit of \cite{PV04,2024arXiv240403343M}.} By employing the Kelvin transform
		\[
		\hat u(x) := \frac{1}{|x|^{N-2}} u\left(\frac{x}{|x|^2}\right),
		\]
		we deduce that \( \hat u \in C^2(\mathbb{R}^N_+) \cap C(\overline{\mathbb{R}^N_+}\setminus\{0\}) \) and $\hat u$ is a solution to
		\[
		\begin{cases}
			-\Delta \hat u = \dfrac{1}{|x|^{N+2}} f(|x|^{N-2} \hat u) &\text{ in } \mathbb{R}^N_+,\\
			\hat u>0 &\text{ in } \mathbb{R}^N_+,\\
			\hat u=0 &\text{ on } \partial\mathbb{R}^N_+\setminus\{0\}.
		\end{cases}
		\]
		\add{Since $u$ is continuous on $\overline{\mathbb{R}^N_+}$ and $u(0)=0$, we have
		\begin{equation}\label{kelvin_decay}
		\hat u(x)\to0 \quad\text{ as } |x|\to+\infty,\ x\in\mathbb{R}^N_+;
		\end{equation}
		indeed, $\hat u(x)\le|x|^{2-N}\sup_{B_1(0)\cap\mathbb{R}^N_+}u$ for $|x|\ge1$ if $N\ge3$, while $\hat u(x)=u(x/|x|^2)$ if $N=2$.}
		
		As in \cite{2024arXiv240403343M}\add{, in this proof only we write $x=(x_1,x')$ with $x'\in\mathbb{R}^{N-1}$, and for $\lambda<0$ we set} $\Sigma_\lambda=\{(x_1,x')\in\mathbb{R}^N_+ \mid x_1<\lambda\}$, $x_\lambda=(2\lambda-x_1,x')$ and $\hat u_\lambda(x) = \hat u(x_\lambda)$. \add{Notice that $\hat u_\lambda$ is a positive solution of
		\[
		-\Delta \hat u_\lambda = \frac{1}{|x_\lambda|^{N+2}} f(|x_\lambda|^{N-2} \hat u_\lambda) \quad\text{ in } \Sigma_\lambda,
		\]
		and that $\hat u_\lambda$ may be singular at the point $0_\lambda:=(2\lambda,0,\dots,0)\in\partial\mathbb{R}^N_+$, the reflection of the origin. We fix $\lambda<0$ and $\tau>0$, and we set
		\[
		w:=(\hat u - \hat u_\lambda - \tau)^+ \ \text{ in } \Sigma_\lambda, \qquad w:=0 \ \text{ in } \mathbb{R}^N_+\setminus\Sigma_\lambda.
		\]
		We claim that the support of $w$ is a compact subset of $\Sigma_\lambda$. Since $\hat u_\lambda>0$, we have $\hat u\ge\tau$ on the support of $w$. Hence, by \eqref{kelvin_decay}, the support of $w$ is bounded. Since $\lambda<0$, the origin does not belong to $\overline{\Sigma_\lambda}$, so $\hat u$ is continuous on $\overline{\Sigma_\lambda}$ and vanishes on $\overline{\Sigma_\lambda}\cap\partial\mathbb{R}^N_+$; therefore the support of $w$ stays away from $\partial\mathbb{R}^N_+$, and in particular from the point $0_\lambda$. Finally, $\hat u=\hat u_\lambda$ on $\{x_1=\lambda\}\cap\mathbb{R}^N_+$, so that the support of $w$ also stays away from this hyperplane. This proves the claim. Consequently, $w$ is a Lipschitz continuous function with compact support in $\Sigma_\lambda$, and both equations can be tested with $w$. Subtracting them, we obtain}
		\[
		\add{\int_{\Sigma_\lambda} |\nabla w|^2 =} \int_{\Sigma_\lambda} \left(\frac{1}{|x|^{N+2}} f(|x|^{N-2} \hat u)-\frac{1}{|x_\lambda|^{N+2}} f(|x_\lambda|^{N-2} \hat u_\lambda)\right) w \le 0,
		\]
		since $\hat u \ge \hat u_\lambda$ on the support of $w$\add{, $f$ is positive and nonincreasing,} and $|x|\ge|x_\lambda|$ in $\Sigma_\lambda$\add{; the last inequality follows from $|x|^2-|x_\lambda|^2=4\lambda(x_1-\lambda)>0$ in $\Sigma_\lambda$. Thus $w\equiv0$. Letting $\tau\to0^+$, we get $\hat u\le\hat u_\lambda$ in $\Sigma_\lambda$ for every $\lambda<0$, and letting $\lambda\to0^-$ gives $\hat u(x_1,x')\le\hat u(-x_1,x')$ whenever $x_1<0$. The function $u(-x_1,x')$ is also a solution to \eqref{main}, and its Kelvin transform is $\hat u(-x_1,x')$; hence the same argument gives the reverse inequality. Therefore $\hat u$, and hence $u$, is symmetric with respect to the hyperplane $\{x_1=0\}$. Since problem \eqref{main} is invariant under translations parallel to $\partial\mathbb{R}^N_+$ and under rotations fixing $e_N$, $u$ is symmetric with respect to every hyperplane orthogonal to $\partial\mathbb{R}^N_+$. As the composition of two reflections with respect to parallel hyperplanes is a translation, $u$ is invariant under all translations parallel to $\partial\mathbb{R}^N_+$. Thus} $u(x)=v(x_N)$, where $v$ is a solution to
		\[
		\begin{cases}
			-v'' = f(v) &\text{ in } \mathbb{R}_+,\\
			v(t) > 0 &\text{ in } \mathbb{R}_+,\\
			v(0) = 0.
		\end{cases}
		\]
		\add{Since $v''=-f(v)<0$, the derivative $v'$ is strictly decreasing. If $v'(t_0)\le0$ for some $t_0>0$, then $v'(t)\le v'(t_0+1)<0$ for $t\ge t_0+1$, and $v$ would become negative, a contradiction. Hence $v'>0$ in $\mathbb{R}_+$. If $v$ were bounded, then $v(t)\nearrow\ell$ for some $\ell\in(0,+\infty)$ and $v''=-f(v)\le-f(\ell)<0$ in $\mathbb{R}_+$, which would force $v'(t)\to-\infty$ as $t\to+\infty$, a contradiction. Hence} $\lim_{t\to+\infty} v(t)=+\infty$. From $-v'' = f(v)$ we deduce
		\[
		\frac{1}{2} (v')^2 + F_1(v) = M_1,
		\]
		where $F_1(s) = \int_1^s f(t) dt$ and $M_1$ is a constant. Sending $t\to+\infty$, we obtain $M_1\ge \int_1^{+\infty} f(t) dt$. In particular, $\int_1^{+\infty} f(t) dt$ is finite. Hence also $F(s) = \int_{s}^{+\infty} f(t) dt$ is finite for all $s>0$. Similar to the proof of Theorem \ref{th:rigidity} we deduce
		\begin{equation}\label{ode3}
			\int_{0}^{v(t)} \frac{ds}{\sqrt{M+F(s)}} = \sqrt{2} t \quad\text{ for all } t\ge0 \text{ and some } M\ge0
		\end{equation}
		and \eqref{ode3} indeed provides a solution to our problem.
	\end{proof}

	\noindent\textbf{Conflict of interest} The author declares no conflict of interest.

	\noindent\textbf{Data Availability} Data sharing not applicable to this article as no datasets were generated or analysed during the current study.

\end{document}